\documentclass[11pt,a4paper]{article}
\usepackage{amsfonts}
\usepackage{amsmath}
\usepackage{amssymb}
\usepackage[english]{babel}
\usepackage[T1]{fontenc}
\usepackage{graphicx}
\usepackage{latexsym}
\usepackage{amstext}
\usepackage{amsthm}
\usepackage{mathrsfs}
\usepackage{graphics}
\usepackage{enumerate}
\usepackage{paralist}
\usepackage{color}
\usepackage{fullpage}
\usepackage{hyperref}

\newtheorem{prop}{Proposition}[section]
\newtheorem{rem}[prop]{Remark}
\newtheorem{lem}[prop]{Lemma}
\newtheorem{theo}[prop]{Theorem}

\numberwithin{equation}{section}

\newcommand{\beq}{\begin{eqnarray}}
\newcommand{\beqq}{\begin{eqnarray*}}
\newcommand{\eeq}{\end{eqnarray}}
\newcommand{\eeqq}{\end{eqnarray*}}

\newcommand{\R}{\mathbb{R}}
\newcommand{\E}{\mathbb{E}}
\renewcommand{\P}{\mathbb{P}}
\newcommand{\B}{\mathbf{B}}

\newcommand{\X}{\mathbb{X}}
\newcommand{\T}{\mathbb{T}}
\newcommand{\F}{\mathbb{F}}

\newcommand{\C}{\mathbb{C}}

\title{Extinction time of non-Markovian self-similar processes,  persistence,  annihilation of jumps   and the Fr\'echet distribution}
\author{ {\sc R.~Loeffen}\thanks{School of Mathematics,
University of Manchester, Manchester M13 9PL, UK
\  E-mail: ronnie.loeffen@manchester.ac.uk}  \:
  {\sc P.~Patie}\thanks{Cornell University, School of Operations Research and Information Engineering,
   220 Rhodes Hall, Ithaca, NY 14853, U.S.A. \  E-mail: pp396@cornell.edu } \: {\sc and}  {\sc M.~Savov}\thanks{Institute of Mathematics and Informatics, Bulgarian Academy of Sciences,  "Akad. Georgi Bonchev" bl. 8, Sofia 1113, Bulgaria \  E-mail: mladensavov@math.bas.bg}}
\date{}

\begin{document}

\maketitle
\begin{abstract}
We start by providing an explicit characterization and analytical properties, including the persistence phenomena, of   the distribution  of the extinction time  $\T$ of a  class of non-Markovian self-similar  stochastic processes with two-sided jumps  that we introduce as a  stochastic time-change of Markovian self-similar processes. 
For a suitably chosen time-changed, 
we observe,  for  classes with two-sided jumps,  the following surprising facts. On the one hand,  all the $\T$'s within a class  have the same law which we  identify in a simple form for all classes and  reduces, in the   spectrally positive case, to the Fr\'echet distribution. On the other hand, each of its distribution corresponds to  the law of an extinction time of a single Markov process without positive jumps, leaving the interpretation that the time-changed has annihilated the effect of positive jumps.  The example of the non-Markovian processes associated to  L\'evy stable processes is detailed.

\end{abstract}

$\vspace{5pt}$

$\vspace{5pt}$

\section{Introduction and main results}
The aim of this paper is to characterize explicitly and derive  analytical properties of  the distribution of the positive random variable
\beq\label{defT}
\T = \inf \{t>0;\: \X_t \leq 0   \}
\eeq
where $\X=(\X_t)_{t\geq 0}$ is the stochastic process defined on a  probability space $(\Omega,\mathcal F,\P)$   by, for $t\geq0$,
\beq\label{LK} \label{eq:def_inv}
\mathbb{X}_t = X_{\lambda_t} \quad \textrm{ where } \lambda_t=\inf \{s>0;\: \chi_s>t \}
\eeq
and $X=(X_t)_{t\geq 0}$
(resp.~$\chi=(\chi_t)_{t\geq 0}$) is a  self-similar of index $\alpha>0$ (resp.~of index $\beta>0$ and a.s.~increasing  with infinite lifetime) Markov process issued from $x>0$ (resp.~issued from $0$).

Our investigation includes the first exit time to the positive half-line $\T$  of the Brownian motion, the Bessel processes and more generally of any non-degenerate stable L\'evy processes,  time-changed by the inverse of a $\beta$-stable subordinator with $0<\beta<1$. The recent years have witnessed the ubiquity of such non-Markovian dynamics in relation to the fractional Cauchy problem, see e.g.~\cite{Toaldo,Orsin,Hairer}, and, also due to  their central role in diverse physical applications within the field of anomalous diffusion, see e.g.~\cite{Meer},  and also for neuronal models for which their long range dependence feature is attractive, see e.g.~\cite{Neuro}. There is a substantial literature devoted to the study and applications of the first passage times of non-Markovian dynamics such as  Gaussian processes and semi-Markov (Markov processes time-changed with the inverse of a subordinator), see  e.g.~\cite{Deng,Dembo, Koren,Guo, Nualart, Toaldo-FPT} including diverse applications in physics.
However, unlike for Markov processes, this literature reveals that the  lack of a general theory makes the analysis of such objects difficult  and  only  very partial statistical information regarding these random variables have been obtained. For instance, for some Gaussian processes and for regularly varying semi-Markov processes,  the persistence probabilities decay rate has been observed, meaning  that the survival probabilities of the first passage time distribution has a power decay which is independent of the state variable, see
the above references \cite{Dembo, Toaldo-FPT} and the references therein. We shall also identify, among different fine properties, the persistence phenomena for the distribution of $\T$, see Theorem \ref{thm1} \ref{it:mth1}) below.

Denoting the law of the process by $\P_x$ when starting from $x>0$, we recall  that  the stochastic process $X$ is said to be self-similar of index $\alpha>0$ (or $\alpha$-self-similar) if the following identity
\begin{equation}\label{eq:defselfsim}
  (X_{c^{\alpha}t},\P_{cx})_{t\geq 0} \stackrel{(d)}{=} (c X_t,\P_{x})_{t\geq 0}
\end{equation}
holds in the sense of finite-dimensional distributions for any $c>0$.  Thus, since $\chi$ is a  $\beta$-self-similar process with a.s.~increasing paths, $\lambda$ has clearly a.s.~continuous and non-decreasing paths and is $\frac{1}{\beta}$-self-similar and non-Markovian. Since $X$ and $\lambda$ are independent, one easily gets that $\X$ is $\frac{\alpha}{\beta}$-self-similar and non-Markovian.

Note that  every jump
of the increasing self-similar Markov process $\chi$ corresponds to a plateau for its continuous inverse $\lambda$. In the physical literature, these periods are interpreted as trapping events in the dynamics of the particle $\X$  and thus slow down  the dynamics of the original particle $X$. For this reason, in the framework of diffusion,  the time-changed process $\X$ is often called a subdiffusion or an anomalous diffusion, see \cite{Meer}.

Next,  we recall that Lamperti \cite{Lamperti} identifies a one-to-one mapping between the class of positive self-similar Markov processes and the class of L\'evy processes. More specifically, one has, under $\P_x, x>0$, that
\begin{equation}\label{eq:def_lamp}
  X_t = x\exp\left(Y_{A_{x^{-\alpha}t}}\right),\quad 0\leq t<\textrm{T}=\inf\{s>0; \: X_s=0\},
\end{equation}
where $A_t = \inf\{s>0;\: \int_{0}^{s}\exp(\alpha Y_u)du>t\}$.
Here $Y=(Y_t)_{t\geq 0}$ as a L\'evy process is a stochastic process with stationary and independent increments with c\`adl\`ag sample paths. Moreover, its law is fully characterized by its characteristic exponent $\Psi(z)  =\log \E[e^{z Y_1}],z \in i\R$, that takes the form
 \beq\label{LK}
\Psi(z)= \Psi(0)+\sigma^2 z^2 + \textrm{a}z + \int_{\R}(e^{z y} -1-yz\mathbb{I}_{\{|y|<1\}})\Pi(dy),
\eeq
in which $\sigma^2,-\Psi(0)\geq 0$ reflect the diffusion coefficient and the killing rate respectively, $\textrm{a} \in\mathbb{R}$, is the coefficient of the linear part and $\Pi$ is the L\'{e}vy measure that characterizes the jumps and satisfies the condition $\int_{\R}(1\wedge |y|^{2})\Pi(dy)<+\infty$ and $\Pi(\{0\})=0$.  We shall also need the analytical Wiener-Hopf factorization of the L\'evy-Khintchine exponent $\Psi_{\alpha}(z)=\Psi(\alpha z)$ of $\alpha Y,$
which is given, for any $z\in i\R$, by
\begin{equation}\label{eq:Wh}
  \Psi_{\alpha}(z)= -\phi_{\alpha}^-(z)\phi_{\alpha}^+(-z)
\end{equation}
where $\phi_{\alpha}^\pm \in \B$, the set of  Bernstein functions, that is
\[ \phi_{\alpha}^\pm(0)\geq 0 \textrm{  and  } \phi_{\alpha}^\pm(u)-\phi_{\alpha}^\pm(0) \textrm{ are of the form  \eqref{LKs} below.}\]
 In order to avoid the trivial situation when $\T=\infty$ $\P_x$-almost surely (a.s.), according to Lamperti, see also \cite[Section 2.2]{Patie-Savov-BG}, it suffices that
 \begin{equation}\label{eq:id_T_exp}
 \textrm{T}  \stackrel{(d)}{=} x^{\alpha} \int_{0}^{\infty}\exp(\alpha Y_t)dt<\infty,
 \end{equation}
 which in turn is equivalent to the assumption
$\phi_{\alpha}^+(0)>0$ in \eqref{eq:Wh}. For this reason, we consider the set \[ \mathcal{N}=\{\Psi \textrm{ of the form } \eqref{LK}; \: \Psi_{\alpha}(z)= -\phi_{\alpha}^-(z)\phi_{\alpha}^+(-z) \textrm{ with } \phi_{\alpha}^+(0)>0 \}.\]
Next, we denote by $\varrho$ the subordinator associated to $\chi$ by the Lamperti mapping \eqref{eq:def_lamp} (replacing $\alpha$ by $\beta$) and its law is characterized by the Bernstein function $\phi(z)= -\log \E[e^{-z \varrho_1}], \Re(z)\geq 0$, which is expressed as
\beq\label{LKs}
\phi(z)=\textrm{d} z   + \int_{0}^{\infty} (1 - e^{-zy}) \vartheta(dy),
\eeq
where $\textrm{d} \geq 0$ and $\vartheta$ is a L\'evy measure such that $\int_{0}^{\infty} (1 \wedge y) \vartheta(dy)<+\infty$. Next, to ensure that the process $\chi$ can start from $0$ which is then viewed as an entrance boundary, one needs in addition that $\int_{0}^{\infty} y\vartheta(dy)<+\infty$ which implies that
\begin{equation}\label{eq:mean}
  \E[\varrho_1] = \phi'(0^+)=\textrm{d}+\int_{0}^{\infty} y \vartheta(dy)<+\infty,
\end{equation}
see e.g.~\cite{Caballero-Chaumont-06-b}. Moreover,  it is easily seen that the Lamperti mapping  yields that $\chi$ has a.s.~increasing paths if and only if the ones of $\varrho $  are also a.s.~increasing. It is well known that the latter holds if  $\phi(\infty)=\infty$ or equivalently either one of the following conditions
 \begin{equation}\label{eq:ass}
  \textrm{d}>0 \textrm{ or/and }  \vartheta(0,1)=\infty,
\end{equation}
holds.
Then, we write
\[ \B_\varrho =\{ \phi \in \B;\: \phi(0)=0,  \eqref{eq:mean} \textrm{ and } \eqref{eq:ass}  \textrm{ hold}\} \]
and we refer to the monograph \cite{Bertoin-96} for a thorough account on L\'evy processes.
Next, for any $\phi \in \B$, we write
\[ \mathfrak{a}_{\phi} = \sup \{u\geq 0; \: |\phi(-u)|<\infty \} \in [0,\infty] \textrm{ and }  \mathfrak{a}^*_{\phi} = \sup \{u\geq 0; \: 0\leq \phi(-u)<\infty \} \in [0,\infty] \]
and note that    $\mathfrak{a}_{\phi}\geq \mathfrak{a}^*_{\phi}$.
We shall also need, for any $\phi \in \B$,  the function $W_\phi$ which is the unique positive-definite function, i.e.~the Mellin transform of a positive measure, that solves the functional equation, for $\Re(z)>-\mathfrak{a}^*_{\phi}$,
\begin{equation}
\label{eq:functional-equation-for-W_phi}
W_\phi(z+1) = \phi(z)W_\phi(z), \quad W_\phi(1) = 1.
\end{equation}
It is easily checked that for any integer $n$, $W_\phi(n+1)=\prod_{k=1}^{n}\phi(k)$. These functions are thoroughly investigated in \cite[Section 4]{Patie-Savov-BG}.

\noindent To summarize, the process $\X$ is  self-similar of index $\frac{\alpha}{\beta}$ starting from $x>0$ and it is non-Markovian with possible upward and downward jumps depending on the support of the L\'evy measure $\Pi$ in \eqref{LK}. Indeed,  the continuity of the paths of $\lambda$ entails that the processes $X$ and $\X$ have  jumps of the same amplitude and direction. Moreover,  from the Lamperti mapping it can be identified uniquely  by the two L\'evy-Khintchine exponents $\Psi_{\alpha}
\in \mathcal{N}$ and $\phi_{\beta} \in \B_{\varrho}$. To emphasize this connection we shall also use the notation
 \beq\label{defT}
\T_{\Psi_{\alpha}}(\phi_{\beta}) = \T = \inf \{t>0;\: \X_t \leq 0   \}
\eeq
and  in the same spirit we may write ${\rm{T}}_{\Psi_\alpha}={\rm{T}}$. As usual, we denote by
$\mathtt{C}^{\infty}_0\!\left(\mathbb{R}^+ \right)$ (resp.~$\mathtt{C}^{k}_0\!\left(\mathbb{R}^+ \right)$)  the space of infinitely (resp.~$k\in \mathbb Z^+$ times) continuously differentiable functions on $\R^+$ vanishing at $\infty$ along with its derivatives. We are now ready to state our first main result.
\begin{theo}\label{thm1}
Let $\Psi \in \mathcal{N}$ and $\phi \in \B_{\varrho}$ and $\alpha,\beta>0$. Then, the following holds.
\begin{enumerate}[1)]
\item For any $x>0$,
 \beq\label{idT_1}
\E_x\left[{\T^{z}_{\Psi_{\alpha}}(\phi_{\beta})}\right]= x^{\frac{\alpha}{\beta}  z}\frac{ \phi_{\alpha}^+(0)}{\beta\phi'(0^+)} \frac{\Gamma(-\frac{z}{\beta})}{W_{\phi_{\beta}}(-\frac{z}{\beta})}\frac{\Gamma(\frac{z}{\beta}+1)W_{\phi_{\alpha}^+}(-\frac{z}{\beta})}{W_{\phi_{\alpha}^{-}}(\frac{z}{\beta}+1)}, \: -\underline{\mathfrak{m}}_{\T}<\Re(z)< \overline{\mathfrak{m}}_{\T},
\eeq
where  $\underline{\mathfrak{m}}_{\T}= \beta(\mathfrak{a}_{\phi^-_{\alpha}} \mathbb{I}_{\{\phi^-_{\alpha}(0)=0\}}+1)\geq \beta$ and $\overline{\mathfrak{m}}_{\T} = \beta (\mathfrak{a}_{\phi_{\beta}} \wedge \mathfrak{a}^*_{\phi^+_{\alpha}})$.
\item The law of $\T_{\Psi_{\alpha}}(\phi_{\beta})$ is absolutely continuous with a density denoted by $f_{\T_{\Psi_{\alpha}}(\phi_{\beta})}$ which has the following smoothness property
\[ f_{\T_{\Psi_{\alpha}}(\phi_{\beta})}\in \mathtt{C}^{\lceil {\rm{N}}\rceil-2}_0\left(\mathbb{R}^+ \right), \]
provided ${\rm{N}}> 1$,  where  $ {\rm{N}}={\rm{N}_{\phi_\beta}}+{\rm{N}_{\Psi_\alpha}} \in [0,\infty],$ \[ {\rm{N}_{\phi_\beta}}=\frac{\vartheta_{\beta}(0,\infty)}{d_{\beta}},\: {\rm{N}_{\Psi_\alpha}}=\frac{\phi^-_{\alpha}(0)+
\vartheta^-_{\alpha}(0,\infty)}{d^-_\alpha}+\frac{v^+_{\alpha}(0^+)}{\phi^+_{\alpha}(0)+\vartheta^+_{\alpha}(0,\infty)}+
\infty\mathbb{I}_{\{d^+_{\alpha}>0\}}, \] and,  $v^+_{\alpha}$ is the density of $\vartheta^+_{\alpha}$, whose existence is justified  in the proof.
\item \label{it:mth1}Let us write simply $\mathfrak{c}_{\alpha}=\mathfrak{a}^*_{\phi_{\alpha}^+}$ and assume that $0 <\mathfrak{c}_{\alpha}<\mathfrak{a}_{\phi_{\beta}}$ with $\Psi_{\alpha}(-\mathfrak{c}_{\alpha})=\phi_{\alpha}^+(-\mathfrak{c}_{\alpha})=0$, $|\Psi_{\alpha}'(-\mathfrak{c}_{\alpha}^+)|<\infty$ and  $\{ b \in \R;  \Psi_{\alpha}(-\mathfrak{c}_{\alpha}+ib)=0\}=\{0\}$ then
\[ \lim_{t\rightarrow \infty}t^{\beta\mathfrak{c}_{\alpha}}\P_x(\T_{\Psi_{\alpha}}(\phi_{\beta})>t)=\frac{\E_x\left[{\T}^{\beta \mathfrak{c}_{\alpha}}_{\Psi_{\alpha}}(\phi_{\beta})\right]}{\mathfrak{c}_{\alpha}\phi_{\alpha}^{+\prime}(-\mathfrak{c}_{\alpha}^+)}\in(0,\infty). \]
Finally, if in addition $|\Psi_{\alpha}''(-\mathfrak{c}_{\alpha}^+)|<\infty$,  $2\leq \lceil {\rm{N}_{\Psi_\alpha}}\rceil<\infty $ (resp.~$\lceil {\rm{N}_{\Psi_\alpha}}\rceil=\infty$ and there exists $k\in\mathbb{Z}^+$ such that $\liminf_{|b|\to\infty}|b|^k\left|\Psi_{\alpha}(-\mathfrak{c}_{\alpha}+ib)\right|>0 $) then for any $n\leq \lceil {\rm{N}_{\Psi_\alpha}}\rceil-2 + (\lceil {\rm{N}_{\phi_\beta}}\rceil -2)\mathbb{I}_{\{\lceil {\rm{N}_{\phi_\beta}}\rceil\geq 2\}}$ (resp.~for any $n\in \mathbb Z^+$)
\[ \lim_{t\rightarrow \infty}t^{\mathfrak{c}_{\alpha} +n+1}(t^{\frac{1}{\beta}-1}f_{\T_{\Psi_{\alpha}}(\phi_{\beta})}(t^{\frac{1}{\beta}}))^{(n)}=\beta(-1)^nC_{\mathfrak{c}_{\alpha}}(n) \frac{\E_x\left[{\T}^{ \mathfrak{c}_{\alpha}}_{\Psi_{\alpha}}(\phi_{\beta})\right]}{\mathfrak{c}_{\alpha}\phi_{\alpha}^{+\prime}(-\mathfrak{c}_{\alpha}^+)} \]
where   $C_{\mathfrak{c}_{\alpha}}(n)=(1+\mathfrak{c}_{\alpha})_{n-\overline{{\rm{N}}}_{\Psi_{\alpha}}} \sum_{k=0}^{\overline{{\rm{N}}}_{\Psi_{\alpha}}} {\overline{{\rm{N}}}_{\Psi_{\alpha}} \choose k}
\frac{\Gamma(\overline{{\rm{N}}}_{\Psi_{\alpha}}-n+1)}{\Gamma(k-n+1)} (-1)^{k}(1+\mathfrak{c}_{\alpha})_{k}$, $\overline{{\rm{N}}}_{\Psi_{\alpha}}=\lceil {\rm{N}_{\Psi_\alpha}}\rceil -2$ and $(1+\mathfrak{c}_{\alpha})_k=\frac{\Gamma(1+\mathfrak{c}_{\alpha} +k)}{\Gamma(1+\mathfrak{c}_{\alpha})}$.
\end{enumerate}
\end{theo}
\begin{rem}\label{rem:M}
	Note that, in item \ref{it:mth1}),  the condition of $-\mathfrak{c}_{\alpha}$ to be the unique  zero of $\Psi_{\alpha}$ on the line $-\mathfrak{c}_{\alpha}+i\R$ is equivalent to the L\'evy process $Y$ being non-lattice, see the discussion prior to \cite[Theorem 2.11]{Patie-Savov-BG}, and hence of $\log \X$ being non-lattice. 
The requirement when $\lceil {\rm{N}_{\Psi_\alpha}}\rceil=\infty$ that there exists $k\in\mathbb{Z}^+$ such that $\liminf_{|b|\to\infty}|b|^k\left|\Psi_{\alpha}(-\mathfrak{c}_{\alpha}+ib)\right|>0 $ is equivalent to $Y$ not being \textit{weak non-lattice}, a new notion introduced  in the aforementioned paper. 
\end{rem}

In order to state our next main result, we define and provide some distributional and analytical properties of  a family of  random variables indexed by the set of Bernstein functions $\B$ that were introduced by the second author in \cite{Patie-Abs} for a subset of $\B$ and  generalize the Fr\'echet one. We recall that the latter  is one of the three non-degenerate
extreme value distribution functions arising as limits of properly renormalized running maxima of i.i.d.~random variables.  We shall need the  notation, borrowed from \cite[Theorem 2.3]{Patie-Savov-BG},  $\overline{\Theta}_{\phi} = \limsup_{|b|\to\infty}
\frac{\int_0^{|b|}\arg \phi(1+iu)du}{|b|} \in [0,\frac{\pi}{2}]$,   $\phi \in \B$.
\begin{prop}\label{cor:sn}
\begin{enumerate}[1)]
  \item  For  any $\phi \in \B$ and $\beta>0$, there exists a positive random variable ${\F}_{\beta}(\phi)$ whose distribution is determined  by
 \beq\label{idT_2}
\E\left[{\F}^{z}_{\beta}(\phi)\right]=   \frac{\Gamma(1-\frac{z}{\beta}) \Gamma(\frac{z}{\beta}+1)}{W_{\phi}(\frac{z}{\beta}+1)}, \: -\underline{\mathfrak{m}}_{\F}<\Re(z)< \overline{\mathfrak{m}}_{\F},\eeq
where $\underline{\mathfrak{m}}_{\F}=\beta(\mathfrak{a}_\phi \mathbb{I}_{\{\phi(0)=0\}}+1)$ and $\overline{\mathfrak{m}}_{{\F}}=\beta$.
\item Moreover, its law  is absolutely continuous with a density $f_{{\F}_{\beta}(\phi)} \in \mathtt{C}^{\infty}_0\!\left(\mathbb{R}^+ \right)$ which admits an analytical extension to  the sector $S_{\phi} =\{z\in \C;\: |\arg(z)|<\pi-\overline{\Theta}_{\phi}\}$ given, for any  $c\in(-\frac{\underline{\mathfrak{m}}_{\F}-1}{\beta},\frac{\overline{\mathfrak{m}}_{\F}+1}{\beta})$, by the Mellin-Barnes integral
\begin{equation}\label{eq:mb}
  f_{{\F}_{\beta}(\phi)}(t)= \frac{\beta}{2\pi i} \int_{c-i\infty}^{c+i\infty}t^{-z\beta}
\frac{\Gamma(-z+1+\frac{1}{\beta})\Gamma(z+1-\frac{1}{\beta})}{W_{\phi}(z+1-\frac{1}{\beta})} dz,
\end{equation}
which expands, for $|t|>\phi^{-\frac{1}{\beta}}(\infty)$, as
\begin{equation}\label{eq:dens}
  f_{{\F}_{\beta}(\phi)}(t)= \beta t^{-\beta-1}\mathrm{I}_{\phi}(e^{i\pi}t^{-\beta})  \textrm{ where } {\mathrm{I}}_{\phi}(z)=\sum_{n=0}^{\infty}\frac{n+1}{\phi(n+1)}\frac{z^n}{W_{\phi}(n+1)}, |z|<\phi(\infty).
\end{equation}
\end{enumerate}
\end{prop}
\begin{rem}
   From \cite[Theorem 4.7]{Patie-Savov-BG}, one can derive a L\'evy-Khintchine type representation for the characteristic function of the real-valued variable $\log {\F}_{\beta}(\phi),\phi \in \B,$ along with sufficient conditions on $\phi$ (or its characteristics) for this variable to be infinitely divisible, see also \cite{Alili-Jadidi-14} for alternative conditions. This remark also applies to $\log \T$ defined in Theorem \ref{thm1}.
\end{rem}
We proceed by establishing some connections between this class of distributions  and  some distributions that have already appeared in the literature.
\begin{itemize}
  \item {\emph{The Fr\'echet distribution}}. When $\phi(u)=u$ above that is $W_{\phi}(n+1)=n!$, then ${\F}_{\beta}={\F}_{\beta}(\phi)$ boils down to the classical Fr\'echet random variable of parameter $\beta>0$, that is
$f_{{\F}_{\beta}}(t)=\beta t^{-\beta -1} e^{-t^{-\beta}}, t>0$. 
  \item \emph{The class of distributions introduced  in \cite{Patie-Abs}}. Let, for some fixed $\alpha>0$, denote by $\psi(u)=(u-\alpha)\phi(u), u\geq 0,$ with $\phi \in \B$, the Wiener-Hopf factorization of the Laplace exponent of   a spectrally negative L\'evy process which is either killed at an independent exponential time or with a negative mean.  Then, it is shown  in \cite[Theorem 2.1]{Patie-Abs}, that $s_{\phi}$ is the density of a positive random variable, where
\[ s_{\phi}( t)=\alpha \phi(\alpha)t^{-2}{\mathcal{I}}_{\psi_{\triangleright\alpha}}(2; e^{i\pi} t^{-1}),\: t>0, \]
and, with the notation of the aforementioned  paper, we used the fact that $\gamma=\alpha$, $\gamma_{\alpha}=1$ and $C_{\gamma}=\psi'(\alpha)=\alpha \phi(\alpha)$, see \cite[Proposition 2.4(2)]{Patie-Abs}, $\psi_{\triangleright\alpha}(u)=\psi(u+\alpha)=u\phi(u+\alpha)$
 and
\[\mathcal{I}_{\psi_{\triangleright\alpha}}(2;\alpha z)=\sum_{n=0}^{\infty}\frac{\Gamma(n+1)(\alpha z)^n}{\prod_{k=1}^n \psi(\alpha (k+1))} = \phi(\alpha)\sum_{n=0}^{\infty}\frac{(n+1)z^n}{\phi_{\alpha}(n+1) W_{\phi_{\alpha}}(n+1)}=\phi(\alpha) \mathrm{I}_{\phi_{\alpha}}(z).\]
Since, it is well-known that $\psi(u)=(u-\alpha)\phi(u)$, as above, if and only if  $\phi_{\triangleright\alpha}(u)=\phi(u+\alpha) \in \B_-$, where  \[\B_-=\{\phi \in \B; \: \textrm{ in } \eqref{LKs} \: \vartheta(dy)=v(y)dy \textrm{ with }  v  \textrm{ non-increasing on } \R_+ \},\]
we get that  for all $\phi \in \B$ such that $\phi_{\triangleright\alpha}\in \B_-$, we have $s_{\phi}(  t) = \frac{1}{\alpha}f_{{\F}_{1}(\phi_{\alpha})}(\alpha t)$. Some illustrative examples are given in \cite{Patie-Abs} and they include the reciprocal of the Gamma and Wright hypergeometric type random variables.
\end{itemize}

We now turn to the statement of the second main result for which we need the following.
We recall, see e.g.~\cite{Patie-SavovEJP}, that the linear operator  $\mathcal{S}_1: f \mapsto \mathcal{S}_1f(u)=\frac{u}{u+1}f(u+1)$ leaves invariant the set $\B$ of Bernstein functions. A slight extension of this transformation has been proposed in \cite[Lemma 10.1.2.]{PS_spectral} and is defined as follows. Let us introduce the subset of Bernstein functions
\begin{equation}\label{eq:defBd}
  \B_1 =\{ \phi \in \B;\: 0 \leq \phi(-u) <\infty \textrm{ for all } u\leq 1\}
\end{equation}
and then for any  $\phi \in \B_1$, we have
\begin{equation}
\mathcal{S}_\phi(u)=\frac{u}{u+1}\phi(u) \in \B.
\end{equation}
To see that, we simply observe that on the one hand for any $\phi \in \B_1$, $\phi_1(u)= \phi(u-1) \in \B$ as it is well-defined,  non-negative and clearly $\phi'(u-1)$ is completely monotone on $\R^+$ and on the other hand $\mathcal{S}_\phi=\mathcal{S}_1\phi_1$.

\begin{theo}\label{cor:sn1}
\begin{enumerate}[1)]
  \item For any $\Psi \in \mathcal{N}_1=\{\Psi \in \mathcal{N} \textrm{ with } \phi_{\alpha}^+ \in \B_1 \textrm{ and } \eqref{eq:ass}  \textrm{ holds} \}$, we have the identity in law
\begin{equation}\label{eq:TF1}
\T_{\Psi_{\alpha}}(\mathcal{S}_{\phi_{\alpha}^+}) \stackrel{(d)}{=} x^{\frac{\alpha}{\beta}}{\F}_{\beta}(\phi_{\alpha}^-).
\end{equation}
\item In particular, for $\beta=1$ and all $\Psi \in \mathcal{N}^-_1=\{\Psi \in \mathcal{N}_1 \textrm{ with } \phi_{\alpha}^- \in \B_- \}$, we have
\begin{equation} \label{eq:idTT}
\T_{\Psi_{\alpha}}(\mathcal{S}_{\phi_{\alpha}^+}) \stackrel{(d)}{=} {\rm{T}} _\psi
\end{equation}
where ${\rm{T}}_\psi  = \inf \{t>0;\: \overline{X}_t \leq 0  \} $ with $\overline{X}=(\overline{X}_{t})_{t\geq 0}$ the spectrally negative $\alpha$-self-similar positive Markov process associated, via the Lamperti mapping, to $\psi(z)=\frac{1}{\alpha}(z-\alpha)\phi^{-}(z) \in \mathcal{N}$.

\item Finally, when $\Psi \in \mathcal{N}^-_1=\{\Psi \in \mathcal{N}_1 \textrm{ with } \phi_{\alpha}^{-}(u)=\alpha u\}$, i.e.~$\X$ is spectrally positive, then
\begin{equation}\label{eq:IdF}
\T_{\Psi_{\alpha}}(\mathcal{S}_{\phi_{\alpha}^+}) \stackrel{(d)}{=} (\alpha x^{\alpha})^{\frac{1}{\beta}} {\F}_{\beta}.
\end{equation}
 Thus, for any $\Psi \in \mathcal{N}^-_1$, $\T_{\Psi_{\alpha}}(\mathcal{S}_{\phi_{\alpha}^+}) $ is a positive self-decomposable variable.
\end{enumerate}

\end{theo}
\begin{rem}
 It is interesting to note in \eqref{eq:TF1} that the passage times $\T_{\Psi_{\alpha}}(\mathcal{S}_{\phi_{\alpha}^+})$ have the same distribution for all $\X$ associated to the same descending ladder height exponent $\phi_{\alpha}^-$, independently of the ascending one and which reduces to a very simple law in the spectrally positive case.
\end{rem}

 \begin{rem}
 In the same vein, it is also surprising to observe, from \eqref{eq:idTT}, that the passage time of the two processes have the same law whereas one process has two-sided jumps whereas the other one has only downward jumps. This leads to the interpretation that the specific time-change annihilates  the role played by the ascending ladder height process of the underlying L\'evy process in the dynamics of the upward jumps of $\X$. It would be  very interesting to obtain  a pathwise explanation of this fact.
\end{rem}


We postpone the proof of these results to the  Section \ref{sec:proof}. We proceed instead  with the description of an example that illustrates the previous  result.


Let us assume that $X=Z^{\frac1b}$ where $Z$
is an  $\mathfrak{a}$-stable L\'evy process, $0<\mathfrak{a}\leq 2$,  with positivity parameter $\rho=\P(Z_1>0)$, killed upon entering into the negative half-line and $0<b\leq 1-\rho$.   This is easily seen to be a positive self-similar  Markov process of index $0<\alpha = \mathfrak{a} b \leq  \mathfrak{a}(1-\rho) $  and it is associated, via the Lamperti mapping, to a L\'evy process with L\'evy-Kintchine exponent  expressed in terms of its Wiener-Hopf factors as follows
\begin{align*}
\Psi_{\alpha } (z)&=-\frac{\Gamma(1+\alpha z)}{\Gamma(1-\mathfrak{a}(1-\rho)+\alpha  z)}\frac{\Gamma(\mathfrak{a}-\alpha  z)}{\Gamma(\mathfrak{a}(1-\rho)-\alpha z)}=-\phi_{\alpha }^-(z)\phi_{\alpha}^+(-z),
\end{align*}
see  \cite[Section 5.1]{Kuznetsov_Pardo_flunctuation}. Note that $\Psi_{\alpha}(0)>0$ unless $\rho=1$ (resp.~or $\rho=0$), which  we exclude as in this case $Z$ is a positive (resp.~a negative) L\'evy process and does not hit the negative (resp.~positive) half-line.
Let us assume for sake of simplicity  that $\beta=1$ and write $\phi=\phi_1$,  then, as $0<\rho<1$,  one easily checks that $\phi_{\alpha}^+ \in \B_1$  and  gets
\begin{equation*}
\phi(u)=\mathcal{S}_{\phi_{\alpha}^+}(u)=\frac{u}{u+1}\frac{\Gamma(\mathfrak{a}+\alpha  u)}{\Gamma(\mathfrak{a}(1-\rho)+\alpha u)}\in \B.
\end{equation*}
Next, we have that $W_{\phi_\alpha^-}$ solves, for $u>0$, the equation
\begin{equation*}
	W_{\phi_\alpha^-}(u+1)=\frac{\Gamma(1+\alpha u)}{\Gamma(1-\mathfrak{a}(1-\rho)+\alpha  u)}W_{\phi_\alpha^-}(u),\,\,W_{\phi_\alpha^-}(1)=1.
\end{equation*}
Recalling that the \textit{Barnes Gamma function} $G$  satisfies the functional equation, for $u,\tau>0$,
\begin{equation*}
G(u+1;\tau)=\Gamma\left(\frac{u}{\tau}\right)G(u;\tau)
\end{equation*}
see e.g.~\cite[(24)]{Kuznetsov_Pardo_flunctuation}, we get
\begin{align*}
W_{\phi_\alpha^-}(u+1)&=\frac{G(1+\frac{1-\mathfrak a(1-\rho)}{\alpha};\frac{1}{\alpha})}{G(\frac{1}{\alpha}+1;\frac{1}{\alpha})}\frac{
G(u+\frac{1}{\alpha}+1;\frac{1}{\alpha})}{G(u+1+\frac{1-\mathfrak a(1-\rho)}{\alpha};\frac{1}{\alpha})}.
\end{align*}
Easy algebra yields that
 \[f_{\T_{\Psi_{\alpha}}(\mathcal{S}_{\phi_{\alpha}^+})}(t)=  x^{\alpha}t^{-2}\mathrm{I}_{G}(e^{i\pi}(x^{-\alpha}t)^{-1}), \: t>0,\]
where  we have set ${\mathrm{I}}_{G}={\mathrm{I}}_{\phi_{\alpha}^{-}}$ with, for any $z \in \C$,
\[{\mathrm{I}}_{G}(z)=\frac{G(\frac{1}{\alpha}+1;\frac{1}{\alpha})}{\alpha G(1+\frac{1-\mathfrak a(1-\rho)}{\alpha};\frac{1}{\alpha})}
\sum_{n=0}^{\infty}\frac{{\Gamma(1-\mathfrak{a}(1-\rho-b)+\alpha  n)}}{\Gamma(\alpha+\alpha n)}\frac{G(n+2+\frac{1-\mathfrak a(1-\rho)}{\alpha};\frac{1}{\alpha})}{
G(n+\frac{1}{\alpha}+2;\frac{1}{\alpha})}z^n.\]
Finally, setting  $\rho=1-\frac{1}{\mathfrak{a}}$ with $1<\mathfrak{a}\leq 2$, that is $Z$ and hence $\X$ is spectrally positive, we obtain indeed that $\phi_{\alpha }^-(z)=\frac{\Gamma(1+\alpha z)}{\Gamma(1-\mathfrak{a}(1-\rho)+\alpha  z)}=\alpha z$ and, from \eqref{eq:IdF}, we get that when $\X$ starts from $\alpha^{-\frac{1}{\alpha}}$ then $\T_{\Psi_{\alpha}}(\mathcal{S}_{\phi_{\alpha}^+})$ has the Fr\'echet distribution of parameter $1$.


\section{Proofs}\label{sec:proof}



Throughout, for a non-negative random variable $X$ we use  the notation
\[\mathcal{M}_{X}(z)=\E[X^{ z}]\]
for at least any $z\in i\R$, the imaginary line, meaning that $\mathcal{M}_{X}(z-1)$ is  its Mellin transform.

\subsection{Proof of Theorem \ref{thm1}}
We start by recalling that  $\chi$ is the increasing self-similar Markov process of index $\beta>0$ starting from $0$ and associated via the Lamperti mapping to the Bernstein function $\phi \in \B_{\varrho}$. We denote by $\lambda=(\lambda_t)_{t\geq 0}$ its continuous right-inverse, see \eqref{eq:def_inv}.
\begin{lem} \label{lem:lambda}
 For any $t>0$ and $\Re(z)>0$,
  \begin{equation}\label{eq:mom-lamb}
    \mathcal{M}_{\lambda_t}(z)=\frac{t^{z\beta}}{\beta \phi'(0^+)} \frac{\Gamma(z)}{W_{\phi_{\beta}}(z)}
  \end{equation}
  where we recall that $\phi_{\beta}(u)=\phi(\beta u) \in \B$. The law of  $\lambda_t$ is absolutely continuous for all $t>0$. Moreover, for any $q \in \C$, 
  \begin{equation}\label{eq:mom-lamb1}
    \E\left[e^{q\lambda_t}\right]=\frac{\beta}{\phi'(0^+)}\overline{\mathrm{I}}_{\phi_{\beta}}(qt^{\beta})
  \end{equation}
  where $\overline{\mathrm{I}}_{\phi_{{\beta}}}(q)=\sum_{n=0}^{\infty}\frac{q^n}{nW_{\phi_{\beta}}(n)}$. Consequently the law of $\lambda_t$ is, for all $t>0$, moment determinate.
  \end{lem}
\begin{proof}
   For any bounded Borel function $f$, we have that
  \begin{eqnarray}
    \E[f(\lambda_t)] &=& \E[f(t^{{\beta}}\lambda_1)] = \int_{0}^{\infty}f(t^{\beta }s)\P(\lambda_1 \in ds)
     = \frac{1}{\beta}\int_{0}^{\infty}s^{-\frac{1}{\beta}-1}f(t^{\beta} s)\P(\chi_1 \in ds^{-\frac{1}{\beta}}) \nonumber\\
  &=&\int_{0}^{\infty}f((t/u)^{\beta})\P(\chi_1 \in du)=   \E\left[f\left(t^{\beta}\chi_1^{-\beta}\right)\right] \label{eq:chil}
  \end{eqnarray}
  where we used the identities $\P(\lambda_1\leq s)=\P(\chi_s \geq 1)=\P(\chi_1 \geq s^{-\frac{1}{\beta}})$.
  Then, according  to \cite[Theorem 2.24]{Patie-Savov-BG}, we deduce  that for any $\Re(z)>0$,
  \begin{eqnarray} \label{eq:mom_chi}
     \mathcal{M}_{\lambda_t}(z) &=& t^{z\beta}    \mathcal{M}_{\chi_1^{\beta}}(-z) = t^{z\beta}\frac{1}{\beta \phi'(0^+)} \frac{\Gamma(z)}{W_{\phi_{\beta}}(z)}
  \end{eqnarray}
 Note that to derive the last identity, we used the fact that the process $\chi^{\beta}=(\chi_t^{\beta})_{t\geq 0}$ is a $1$-self-similar increasing Markov process associated to the subordinator $\beta \varrho $ whose Laplace exponent is  $\phi_{\beta}$. Next, since $\phi \in \B_{\varrho}$, one can apply \cite[Theorem 1(iii)]{Bertoin-Yor} to get that, for any bounded Borel function $f$,
 \[ \E[f(\chi_1)] = \frac{1}{\beta \phi'(0^+)} \E\left[\frac{1}{I}f\left(\frac{1}{I}\right)\right] \]
 where $I=\int_0^{\infty} e^{-\beta \varrho_t} dt$. Since from \cite{Bertoin2008} the distribution of $I$ is known to be absolutely continuous, we deduce, using also \eqref{eq:chil}, the same property for the law of $\lambda_t$ for any $t>0$.
Finally, by an expansion of  the exponential function combined with an application of a standard Fubini argument, of  the previous identity and of the recurrence relation for the gamma function,  one gets
  \begin{equation*}
    \E\left[e^{q\lambda_t}\right]=\sum_{n=0}^{\infty}\E[\lambda^{ n}_t]\frac{q^n}{n!}=\frac{1}{\beta \phi'(0^+)} \sum_{n=0}^{\infty}   \frac{1}{n}\frac{(t^\beta q)^n}{W_{\phi_{\beta}}(n)},
  \end{equation*}
  where, by using the functional equation \eqref{eq:functional-equation-for-W_phi}, the series is easily checked to be absolutely convergent on $|q|t^{\beta}<\phi(\infty)$. Since $\phi \in \B_{\varrho}$ then $\phi(\infty)=\infty$ and hence $\overline{\mathrm{I}}_{\phi_{{\beta}}}$ defines an entire function.  The last claim is then immediate.
\end{proof}




\begin{prop}\label{prop:thm1}
Let $\Psi \in \mathcal{N}$ and $\phi \in \B_{\varrho}$. For any $x > 0$, we have $\P_x$ a.s.
 \beq\label{idT2}
\T_{\Psi_{\alpha}}(\phi_{\beta})  \stackrel{(d)}{=} \chi_{{\rm{T}}_{\Psi_{\alpha}}}
\eeq
where we recall that  ${\rm{T}}_{\Psi_{\alpha}}  = \inf \{t>0;\: X_t \leq 0  \}$ where $X$ is an $\alpha$-self-similar positive Markov process associated to $\Psi$ via the Lamperti mapping. Consequently
 \beq\label{idT1}
\T_{\Psi_{\alpha}}(\phi_{\beta}) \stackrel{(d)}{=} \chi_1 \times {\rm{T}}_{\Psi_{\alpha}} ^{\frac{1}{\beta}}
\eeq
where $\times$ stands for the product of two independent random variables.
\end{prop}
\begin{proof}
First, recall  that  $t \mapsto \lambda_t $  is a.s.~continuous with for any $t\geq 0$, $\lambda_{\chi_t}=t$ a.s., and thus for any $x > 0 $, we have $\P_x$ a.s.
\beq\label{defT1}
\T_{\Psi_{\alpha}}(\phi_{\beta}) = \inf \{t>0;\: \X_t\leq 0 \} = \inf \{\chi_t>0;\: X_t \leq 0 \} = \chi_{ \inf \{t>0;\: X_{t} \leq  0 \}} = \chi_{{\rm{T}}_{\Psi_{\alpha}}},
\eeq
which provides \eqref{idT2} while \eqref{idT1} follows immediately by using the independence of $\chi$ and ${\rm{T}}_{\Psi_{\alpha}}$, and the fact that $\chi$ is a self-similar process of index $\beta$. 
\end{proof}
\subsubsection{End of the proof of Theorem \ref{thm1}}
Let $\Psi \in \mathcal{N}$ and $\phi \in \B_{\varrho}$ and write simply here $\T$ for $\T_{\Psi_{\alpha}}(\phi_{\beta})$ and ${\rm{T}}$ for ${\rm{T}}_{\Psi_\alpha}$.
By independence of the variables ${\rm{T}} $ and $\chi_1$, and recalling that $\chi_1^{\beta}$ is a $1$-self-similar increasing Markov process associated to the subordinator $\beta\varrho$ whose Laplace exponent is  $\phi_{\beta}(\cdot)=\phi(\beta\cdot)$, we get, for any $-1<\Re(z)<0 $, that
\begin{equation*}
	\begin{split}
		\mathcal{M}_{\T}( \beta z)&=  \mathcal{M}_{\chi^\beta_1}(z)\mathcal{M}_{{\rm{T}} }\left(z\right)
		=\frac{1}{\beta\phi'(0^+)} \frac{\Gamma(-z)}{W_{\phi_{\beta}}(-z)}x^{\alpha z}\phi_{\alpha}^+(0)\frac{\Gamma\left(z+1\right)}{W_{\phi_{\alpha}^{-}}\left(z+1\right)}W_{\phi_{\alpha}^+}\left(-z\right),
	\end{split}	
	\end{equation*}
where for the second identity we have used \eqref{eq:mom_chi}, the identity  \eqref{eq:id_T_exp}, that is ${\rm{T}}  \stackrel{(d)}{=} x^{\alpha} \int_{0}^{\infty}\exp(\alpha Y_t)dt$   under $\P_x, x>0$, and the expression of the Mellin transform  of the so-called exponential functional which is found in \cite[Theorem 2.4]{Patie-Savov-BG}. The expression \eqref{idT_1} follows then readily and another change of variable yields that the Mellin transform of $\T$ (under $\P_x,x>0$) is given by
\beq\label{idT_1}
\mathcal{M}_{\T}(z-1)= x^{\frac{\alpha}{\beta}  (z-1)}\frac{ \phi_{\alpha}^+(0)}{\beta\phi'(0^+)} \frac{\Gamma(-\frac{z}{\beta}+\frac{1}{\beta})}{W_{\phi_{\beta}}(-\frac{z}{\beta}+\frac{1}{\beta})}
\frac{\Gamma(\frac{z}{\beta}+1-\frac{1}{\beta})W_{\phi_{\alpha}^+}(-\frac{z}{\beta}+\frac{1}{\beta})}{W_{\phi_{\alpha}^{-}}(\frac{z}{\beta}+1-\frac{1}{\beta})}.\eeq
Then,
 \cite[Theorem 2.3(2.13)]{Patie-Savov-BG} yields that the mappings $z\mapsto \frac{\Gamma(-\frac{z}{\beta}+\frac{1}{\beta})}{W_{\phi_{\beta}}(-\frac{z}{\beta}+\frac{1}{\beta})}$, $z\mapsto \frac{\Gamma(\frac{z}{\beta}+1-\frac{1}{\beta})}{W_{\phi_{\alpha}^{-}}(\frac{z}{\beta}+1-\frac{1}{\beta})}$,  and $z\mapsto W_{\phi_{\alpha}^+}(-\frac{z}{\beta}+\frac{1}{\beta})$  are analytical on  $\Re(z)< \beta \mathfrak{a}_{\phi_{\beta}}+1$ (recall that $\phi_{\beta}(0)=0$),  $\Re(z)>-\beta (\mathfrak{a}_{\phi^-_{\alpha}} \mathbb{I}_{\{\phi^-_{\alpha}(0)=0\}}+1)+1$ and  $\Re(z) <\beta \mathfrak{a}^*_{\phi^+_{\alpha}}+1$ respectively. Putting pieces together and changing variables gives that $\mathcal{M}_{\T}(z) $ is analytical on the strip $\{z\in \C;\: -\beta (\mathfrak{a}_{\phi^-_{\alpha}} \mathbb{I}_{\{\phi^-_{\alpha}(0)=0\}}+1)<\Re(z)< \beta ( \mathfrak{a}_{\phi_{\beta}}\wedge \mathfrak{a}^*_{\phi^+_{\alpha}})\}$.
Next, the fact that the law of $\T$ is absolutely continuous with density $f_{\T}$ follows from  the absolute continuity of the law of the random variable $\int_{0}^{\infty}\exp(\alpha Y_t)dt$,  see \cite{Bertoin2008}, combined with the  identities \eqref{idT1} and \eqref{eq:id_T_exp}.

To understand the smoothness of $f_\T$ we investigate the decay along complex lines of the terms of \eqref{idT_1}.
 From \cite[Theorem 2.4(3)]{Patie-Savov-BG} for any $-\frac{1}{\beta}<a<0$ and for any $p<\textrm{N}_{\phi_{\beta}}=\frac{\vartheta_{\beta}(0,\infty)}{d_{\beta}}$, with $\textrm{N}_{\phi_{\beta}}=\infty$ provided $\vartheta_{\beta}(0,\infty)=\infty$ or $d_{\beta}=0$, we have
\begin{equation}\label{eq:lim}
\begin{split}
&\lim\limits_{|b|\to\infty}|b|^p\left|\frac{\Gamma(-\frac{a}{\beta}-i\frac{b}\beta)}{W_{\phi_{\beta}}(-\frac{a}{\beta}-i\frac{b}\beta)}\right|=0
\end{split}
\end{equation}	
 whereas for any $p>\textrm{N}_{\phi_{\beta}}$
 \begin{equation}\label{eq:lim2}
 \begin{split}
 &\lim\limits_{|b|\to\infty}|b|^p\left|\frac{\Gamma(-\frac{a}{\beta}-i\frac{b}\beta)}{W_{\phi_{\beta}}(-\frac{a}{\beta}-i\frac{b}\beta)}\right|=\infty,
 \end{split}
 \end{equation}	
 which is possible if and only if $\textrm{N}_{\phi_{\beta}}<\infty$. Also, from \cite[Theorem 2.3]{Patie-Savov-BG},  for any
 \[p<\textrm{N}_{\Psi} \textrm{ where N}_{\Psi}=\frac{\phi^-_{\alpha}(0)+\vartheta^-_{\alpha}(0,\infty)}{d^-_\alpha}+\frac{v^+_{\alpha}(0^+)}{\phi^+_{\alpha}(0)+\vartheta^+_{\alpha}(0,\infty)}\]
 we have that, for any $-\frac{1}{\beta}<a<0$,
 \begin{equation}\label{eq:lim1}
 \begin{split}
 &\lim\limits_{|b|\to\infty}|b|^p\left|\mathcal{M}_{\textrm{T}}\left(\frac{a+ib}{\beta}\right)=\frac{\Gamma\left(1+\frac{a}\beta+i\frac{b}\beta\right)}{W_{\phi_{\alpha}^{-}}\left(1+\frac{a}\beta+i\frac{b}\beta\right)}W_{\phi_{\alpha}^+}\left(\frac{-a-ib}{\beta}\right)\right|=0
 \end{split}
 \end{equation}
 with $\textrm{N}_{\Psi}=\infty $
 unless $\Psi(z)-\textrm{a}z$ is bounded with $\textrm{a}<0$, that is  $Y$ is a compound Poisson processes with a strictly negative drift $\textrm{a}$,  in which case, for any
 $p>\textrm{N}_{\Psi} \in[0,\infty)$,
 \begin{equation}\label{eq:lim3}
 \begin{split}
 &\lim\limits_{|b|\to\infty}|b|^p\left|\mathcal{M}_{\textrm{T}}\left(\frac{a+ib}{\beta}\right)=\frac{\Gamma\left(1+\frac{a}\beta+i\frac{b}\beta\right)}{W_{\phi_{\alpha}^{-}}\left(1+\frac{a}\beta+i\frac{b}\beta\right)}W_{\phi_{\alpha}^+}\left(\frac{-a-ib}{\beta}\right)\right|=\infty.
 \end{split}
 \end{equation}

Collecting the decay in \eqref{eq:lim} and \eqref{eq:lim1}  we therefore get that, for any $a\in(-\beta (\mathfrak{a}_{\phi^-_{\alpha}} \mathbb{I}_{\{\phi^-_{\alpha}(0)=0\}}+1),\beta ( \mathfrak{a}_{\phi_{\beta}}\wedge \mathfrak{a}^*_{\phi^+_{\alpha}}))$ and any $p<{\rm{N}}$ (resp.~$p>{\rm{N}}$) where we recall that
${\rm{N}}=\textrm{N}_{\phi_{\beta}}+\textrm{N}_{\Psi}\in [0,\infty]$
we have that
\begin{equation*}
\lim_{|b|\to \infty}|b|^p|\mathcal{M}_{\T}(a+ib)|=0  \quad (\textrm{resp.~} \lim_{|b|\to \infty}|b|^p|\mathcal{M}_{\T}(a+ib)|=\infty).
\end{equation*}
If ${\rm{N}}>1$ by Mellin inversion we deduce that $f_{\T}\in\mathtt{C}^{\lceil {\rm{N}}\rceil-2}_0\left(\mathbb{R}^+ \right)$, see \cite{Mellin} or \cite[(7.10)]{Patie-Savov-BG}. The sufficient conditions for ${\rm{N}}=\infty$ are easily derived.
 The fact that $\vartheta^+_{\alpha}(dy)=v^+_{\alpha}(y)dy, y>0,$ with  $v^+_{\alpha}(0^+)\in(0,\infty)$ follows from \cite[Proposition B.2]{Patie-Savov-BG}.
Next,  from \eqref{idT1}, we get that for all $t>0$,
\[ \P_x(\T>t) =\int_0^{\infty}\P_x({\rm{T}}^{\frac{1}{\beta}}>t/r)f_{\chi_1}(r)dr, \]
where $f_{\chi_1}(t)dt=\P(\chi_1 \in dt),t>0,$ whose existence is justified in the proof of Proposition \ref{lem:lambda}.
Recall, from \eqref{eq:mom_chi}, that
\[ \mathcal{M}_{\chi_1}(z-1)= \frac{1}{\beta\phi'(0^+)} \frac{\Gamma(-\frac{z}{\beta}+\frac{1}{\beta})}{W_{\phi_{\beta}}(-\frac{z}{\beta}+\frac{1}{\beta})},\: \Re(z)<\overline{\mathfrak{m}}_{\chi} = \beta \mathfrak{a}_{\phi_{\beta}} +1,\]
and, from \cite[Theorem 2.11(2)]{Patie-Savov-BG}, under the conditions of the claim, one gets  that
\begin{equation}
\lim_{t\to \infty} t^{\beta \mathfrak{c}_{\alpha}} \P_x({\rm{T}}^{\frac{1}{\beta}}>t)=
\frac{\E_x[{\rm{T}}^{\mathfrak{c}_{\alpha}}]}{\mathfrak{c}_{\alpha}\phi_{\alpha}^{+\prime}(-\mathfrak{c}_{\alpha}^+)},\end{equation} where we have used that $\E_x[{\rm{T}}^{\mathfrak{c}_{\alpha}}]= x^{\alpha} \E\left[\left(\int_{0}^{\infty}\exp(\alpha Y_t)dt\right)^{\mathfrak{c}_{\alpha}}\right]$ and the expression of the moments of the latter functional   in \cite[Theorem 2.4]{Patie-Savov-BG}. Plainly
\begin{equation}
t \mapsto t^{\beta \mathfrak{c}_{\alpha} +1}\P_x({\rm{T}}^{\frac{1}{\beta}}>t) \textrm{ is  bounded on  $(0,a]$ for any } a>0.
\end{equation} Hence, one has all the conditions of \cite[Theorem 4.1.6]{Bingham} to conclude  that
\[\lim_{t\to \infty} t^{\beta \mathfrak{c}_{\alpha}} \P_x({\T}>t)=
\mathcal{M}_{\chi_1}(\beta\mathfrak{c}_{\alpha})
\frac{\E_x[{\rm{T}}^{\mathfrak{c}_{\alpha}}]}{\mathfrak{c}_{\alpha}\phi_{\alpha}^{+\prime}(-\mathfrak{c}_{\alpha}^+)}
=\frac{\E_x[{\T}^{\mathfrak{c}_{\alpha}}]}{\mathfrak{c}_{\alpha}\phi_{\alpha}^{+\prime}(-\mathfrak{c}_{\alpha}^+)}\]
which is the first claim of item \ref{it:mth1}).
Then, \eqref{idT1} raised to the power $\beta$ yields that
\[ f_{\T^{\beta}}(t) =\int_0^{\infty}f_{\textrm{T}}(t/r)f_{\chi^{\beta}_1}(r)\frac{dr}{r}, \]
where we have set  $f_{\textrm{T}}(t)dt=\P_x(\textrm{T} \in dt)$ and $ f_{\T^{\beta}}(t)= \frac{t^{\frac{1}{\beta}-1}}{\beta}f_{\T}(t^{\frac{1}{\beta}})$. Next,  for any $n\leq \lceil {\rm{N}} \rceil -2$, from the general theory of Mellin transform, see \cite[11.7]{Mellin}, one obtains that the Mellin transform of  $f^{(n)}_{\T^{\beta}} \in \mathtt{C}^{ \lceil {\rm{N}} \rceil-2-n}_0(\R^+)$, is given, a priori in the sense of distribution, for any $z$ in the strip $S_{\T^{\beta},n}=\{z\in \C;\: -\mathfrak{a}_{\phi^-_{\alpha}} \mathbb{I}_{\{\phi^-_{\alpha}(0)=0\}}+n<\Re(z)< 1+n+ \mathfrak{a}_{\phi_{\beta}}\wedge \mathfrak{a}^*_{\phi^+_{\alpha}}\}$,  by
 \begin{eqnarray}\label{eq:melder}
\mathcal{M}_{f^{(n)}_{\T^{\beta}}}(z-1)&=& \frac{\Gamma(z)}{\Gamma(z-n)}\mathcal{M}_{\T^{\beta}}(z-1-n) \nonumber \\ &=&\frac{\Gamma(z)}{\Gamma(z-n)} \mathcal{M}_{\textrm{T}}(z-1-n)\mathcal{M}_{\chi^{\beta}_1}(z-1-n) 
 \end{eqnarray}
 Since,  the Stirling formula, see \cite[(6.10)]{Patie-Savov-BG}, yields that, for any $z=a+ib$, and large $b$,
 \begin{equation}\label{eq:stir}
  \left| \frac{\Gamma(z)}{\Gamma(z-n)} \right| \leq C |b|^{n}
 \end{equation}
 for some $C>0$, we get that, for all $\epsilon>0$ and $z=a+ib \in S_{\T^{\beta},n}$,
 \begin{equation}\label{}
  \left| \mathcal{M}_{f^{(n)}_{\T^{\beta}}}(z-1)\right| \leq C |b|^{n-\rm{N}-\epsilon}.
 \end{equation}
 Thus,as $n\leq \lceil {\rm{N}} \rceil -2, z\mapsto \mathcal{M}_{f^{(n)}_{\T^{\beta}}}(z-1)$ is integrable along imaginary lines and hence it is the Mellin transform in the classical sense of $f^{(n)}_{\T^{\beta}}$. Now, set  $\overline{{\rm{N}}}_{\Psi_\alpha}=\lceil {\rm{N}}_{\Psi_\alpha} \rceil -2$ and $\overline{{\rm{N}}}_{\phi_\beta}=\lceil {\rm{N}}_{\phi_\beta} \rceil -2$, and we shall consider the two cases $n\leq \overline{{\rm{N}}}_{\Psi_\alpha}$ and $\overline{{\rm{N}}}_{\Psi_\alpha}< n \leq \overline{{\rm{N}}}_{\Psi_\alpha}+\overline{{\rm{N}}}_{\phi_\beta}\mathbb{I}_{\{\overline{{\rm{N}}}_{\phi_\beta}> 0\}}$. Let us assume first that $n\leq \overline{{\rm{N}}}_{\Psi_\alpha}$ and proceeding as above, combining \eqref{eq:lim3} and \eqref{eq:stir}, we get that, for all $\epsilon>0$ and $z=a+ib \in S_{\textrm{T},n}=\{z\in \C;\: -\mathfrak{a}_{\phi^-_{\alpha}} \mathbb{I}_{\{\phi^-_{\alpha}(0)=0\}}+n<\Re(z)< 1+n+  \mathfrak{c}_{\alpha}\}$,
  \begin{eqnarray}
\left| \mathcal{M}_{f^{(n)}_{\textrm{T}}}(z-1)\right|  &=& \left|\frac{\Gamma(z)}{\Gamma(z-n)}\mathcal{M}_{\textrm{T}}(z-1-n)\right| \leq C |b|^{n-{\rm{N}}_{\Psi_\alpha}-\epsilon}
 \end{eqnarray}
 and deduce that  $f^{(n)}_{\textrm{T}} \in \mathtt{C}^{\overline{{\rm{N}}}_{\Psi_\alpha}-n}_0(\R^+)$.
Moreover,  by the Mellin inversion formula, 
we have, that for any $ c \in ( -\mathfrak{a}_{\phi^-_{\alpha}} \mathbb{I}_{\{\phi^-_{\alpha}(0)=0\}}-2, \mathfrak{c}_{\alpha}-1)$ and all $t>0$,
\begin{eqnarray}\label{eq:MBc}
 |f^{(n)}_{\textrm{T}}(t)|&=&\left|\frac{(-1)^n}{2\pi i} \int_{c+n-i\infty}^{c+n+i\infty} t^{-z} \mathcal{M}_{f^{(n)}_{\textrm{T}}}(z-1)
 dz  \right| \nonumber \\ &\leq & C t^{-c-n} \int_{c+n-i\infty}^{c+n+i\infty} \left|
 \mathcal{M}_{f^{(n)}_{\textrm{T}}}(z-1) \right| dz \nonumber \\ & \leq & \bar{C} t^{-c-n},
\end{eqnarray}
for some $C,\bar{C}>0$. 
Appealing again to \cite[Theorem 2.11(2)]{Patie-Savov-BG}, under the conditions of the claim, we get  that, for any $n\leq \overline{{\rm{N}}}_{\Psi_\alpha}$,
\begin{equation} \label{eq:asympt-fT}
\lim_{t\rightarrow \infty}t^{\mathfrak{c}_{\alpha} +n+1}f^{(n)}_{\textrm{T}}(t)=(-1)^{n}(1+\mathfrak{c}_{\alpha})_{n}\frac{\E_x\left[{\textrm T}^{ \mathfrak{c}_{\alpha}}\right]}{\mathfrak{c}_{\alpha}\phi_{\alpha}^{+\prime}(-\mathfrak{c}_{\alpha}^+)}. \end{equation}
 On the other hand, we deduce, from \eqref{eq:MBc}, that the mapping $t\mapsto t^{\mathfrak{c}_{\alpha}+n +3} f^{(n)}_{\textrm{T}}(t)$ is bounded on any interval $(0,a],a>0$. Then,  the mapping
$z\mapsto \mathcal{M}_{\chi_1^{\beta}}(z-1-n)$, being  analytical on the half-plane $ \Re(z)<  \mathfrak{a}_{\phi_{\beta}} +\frac{1}{\beta}-n$,  is the Mellin transform of the measurable function $t^{-n}f_{\chi^{\beta}_1}(t)$. Thus,  we observe, from \eqref{eq:melder}, the Mellin convolution which translates, for any $t>0$, as
\[ f^{(n)}_{\T^{\beta}}(t) =\int_0^{\infty}f^{(n)}_{\textrm{T}}(t/r)f_{\chi^{\beta}_1}(r)\frac{dr}{r^{n+1}}. \]
Hence, we can use \cite[Theorem 4.1.6]{Bingham} to get
\[ \lim_{t\rightarrow \infty}t^{\mathfrak{c}_{\alpha} +n+1} f^{(n)}_{\T^{\beta}}(t)=(-1)^{n}(1+\mathfrak{c}_{\alpha})_{n}\frac{\E_x\left[{\textrm T}^{ \mathfrak{c}_{\alpha}}\right]}{\mathfrak{c}_{\alpha}\phi_{\alpha}^{+\prime}(-\mathfrak{c}_{\alpha}^+)}
\mathcal{M}_{\chi_1^{\beta}}(\mathfrak{c}_{\alpha})=(-1)^n(1+\mathfrak{c}_{\alpha})_n\frac{\E_x\left[{\T}^{\beta \mathfrak{c}_{\alpha}}_{\Psi_{\alpha}}(\phi_{\beta})\right]}{\mathfrak{c}_{\alpha}\phi_{\alpha}^{+\prime}(-\mathfrak{c}_{\alpha}^+)} \]
which provides the statement for $n\leq \overline{{\rm{N}}}_{\Psi_\alpha}$. Finally, assume that $\overline{{\rm{N}}}_{\Psi_\alpha}< n \leq \overline{{\rm{N}}}_{\Psi_\alpha}+\overline{{\rm{N}}}_{\phi_\beta}\mathbb{I}_{\{\overline{{\rm{N}}}_{\phi_\beta}> 0\}}$ and note from \eqref{eq:melder} that, for any $z=a+ib \in S_{\T^{\beta},n}$,
\begin{eqnarray}\label{eq:melder}
\mathcal{M}_{f^{(n)}_{\T^{\beta}}}(z-1) &=&\frac{\Gamma(z)}{\Gamma(z-n)} \mathcal{M}_{\textrm{T}}(z-1-n)\mathcal{M}_{\chi^{\beta}_1}(z-1-n) \nonumber \\
&=&\frac{\Gamma(z)}{\Gamma(z-\overline{{\rm{N}}}_{\Psi_\alpha})}\mathcal{M}_{{\rm{T}} }\left(z-1-n\right)
\frac{\Gamma(z-\overline{{\rm{N}}}_{\Psi_\alpha})}{\Gamma(z-n)} \mathcal{M}_{\chi^\beta_1}(z-1-n).
 \end{eqnarray}
 Then, we recall that (the first identity serves as setting up a notation) \[\mathcal{M}_{\overline{f}^{(\overline{{\rm{N}}}_{\Psi_\alpha})}_{(\textrm T,n)}}(z-1)= \frac{\Gamma(z)}{\Gamma(z-\overline{{\rm{N}}}_{\Psi_\alpha})}\mathcal{M}_{{\rm{T}} }\left(z-1-n\right)=\frac{\Gamma(z)}{\Gamma(z-\overline{{\rm{N}}}_{\Psi_\alpha})}\mathcal{M}_{{\rm{T}} }\left(z+\overline{{\rm{N}}}_{\Psi_\alpha}-n-1-\overline{{\rm{N}}}_{\Psi_\alpha}\right) \] is the Mellin transform of the function $\overline{f}^{(\overline{{\rm{N}}}_{\Psi_\alpha})}_{(\textrm T,n)}(t)=(t^{\overline{{\rm{N}}}_{\Psi_\alpha}-n}f_{\textrm T}(t))^{(\overline{{\rm{N}}}_{\Psi_\alpha})} \in \mathtt C_0(\R^+)$. We also identify, by uniqueness of the Mellin transform and combining the estimates \eqref{eq:lim2} and \eqref{eq:stir},
 \[\mathcal{M}_R(z)=\frac{\Gamma(z-\overline{{\rm{N}}}_{\Psi_\alpha})}{\Gamma(z-n)} \mathcal{M}_{\chi^\beta_1}(z-1-n)=\frac{\Gamma(z-\overline{{\rm{N}}}_{\Psi_\alpha})}{\Gamma(z-\overline{{\rm{N}}}_{\Psi_\alpha}-(n-\overline{{\rm{N}}}_{\Psi_\alpha}))} \mathcal{M}_{\chi^\beta_1}(z-\overline{{\rm{N}}}_{\Psi_\alpha}-1-(n-\overline{{\rm{N}}}_{\Psi_\alpha}))\] as the Mellin transform of the continuous function $t^{-\overline{{\rm{N}}}_{\Psi_\alpha}}f^{(n-\overline{{\rm{N}}}_{\Psi_\alpha})}_{\chi^\beta_1}(t)$, as by assumption $1\leq n-\overline{{\rm{N}}}_{\Psi_\alpha}\leq \overline{{\rm{N}}}_{\phi_{\beta}}$. As the gamma function has simple poles at $-n,n=0,1,\ldots$, we have  by analytical continuation that $z\mapsto \mathcal{M}_R(z)$ is analytical on the half-plane $\Re(z)<  \mathfrak{a}_{\phi_{\beta}} +1+n$.  We also  deduce, by Mellin convolution, that for any $t>0$,
\begin{eqnarray*}\label{eq:MB}
 f^{(n)}_{{\T^{\beta}}}(t)=\int_0^{\infty}\overline{f}^{(\overline{{\rm{N}}}_{\Psi_\alpha})}_{\textrm{T}}(t/r)r^{-\overline{{\rm{N}}}_{\Psi_\alpha}-1}
 f^{(n-\overline{{\rm{N}}}_{\Psi_\alpha})}_{\chi^\beta_1}(r)dr.
\end{eqnarray*}
Next combining the identity
\[\overline{f}^{(\overline{{\rm{N}}}_{\Psi_\alpha})}_{\textrm T}(t)=\sum_{k=0}^{\overline{{\rm{N}}}_{\Psi_\alpha}} {\overline{{\rm{N}}}_{\Psi_\alpha} \choose k}
\frac{\Gamma(\overline{{\rm{N}}}_{\Psi_\alpha}-n+1)}{\Gamma(k-n+1)} t^{k-n} f^{(k)}_{\textrm T}(t)\]
 with the estimate \eqref{eq:asympt-fT} yields that
\begin{eqnarray*}
\lim_{t\rightarrow \infty}t^{\mathfrak{c}_{\alpha} +n+1}\overline{f}^{(\overline{{\rm{N}}}_{\Psi_\alpha})}_{(\textrm T,n)}(t)&=&\sum_{k=0}^{\overline{{\rm{N}}}_{\Psi_\alpha}} {\overline{{\rm{N}}}_{\Psi_\alpha} \choose k}
\frac{\Gamma(\overline{{\rm{N}}}_{\Psi_\alpha}-n+1)}{\Gamma(k-n+1)} \lim_{t\rightarrow \infty}t^{\mathfrak{c}_{\alpha}+1+k} f^{(k)}_{\textrm T}(t) \nonumber \\
&=&\frac{\E_x\left[{\textrm T}^{ \mathfrak{c}_{\alpha}}\right]}{\mathfrak{c}_{\alpha}\phi_{\alpha}^{+\prime}(-\mathfrak{c}_{\alpha}^+)} \sum_{k=0}^{\overline{{\rm{N}}}_{\Psi_\alpha}} {\overline{{\rm{N}}}_{\Psi_\alpha} \choose k}
\frac{\Gamma(\overline{{\rm{N}}}_{\Psi_\alpha}-n+1)}{\Gamma(k-n+1)} (-1)^{k}(1+\mathfrak{c}_{\alpha})_{k}.
\end{eqnarray*}
Moreover,  by the Mellin inversion formula, 
we have, that for any $ c \in (n
+\mathfrak{a}_{\phi^-_{\alpha}} \mathbb{I}_{\{\phi^-_{\alpha}(0)=0\}}-2, n+\mathfrak{c}_{\alpha}-1)$ and all $t>0$,
\begin{eqnarray}\label{eq:MBc}
 |\overline{f}^{(\overline{{\rm{N}}}_{\Psi_\alpha})}_{(\textrm T,n)}(t)|&=&\left|\frac{(-1)^n}{2\pi i} \int_{c-i\infty}^{c+i\infty} t^{-z} \mathcal{M}_{\overline{f}^{(\overline{{\rm{N}}}_{\Psi_\alpha})}_{(\textrm T,n)}}(z-1)
 dz  \right| \nonumber 
 \leq  \bar{C} t^{-c},
\end{eqnarray}
for some $\bar{C}>0$.
Since $t \mapsto t^{\mathfrak{c}_{\alpha}+n+3}\overline{f}^{(\overline{{\rm{N}}}_{\Psi_\alpha})}_{(\textrm T,n)}(t)$ is bounded on any interval $(0,a],a>0$, one can use again
\cite[Theorem 4.1.6]{Bingham} to get
\begin{eqnarray*}
\lim_{t\rightarrow \infty}t^{\mathfrak{c}_{\alpha} +n+1}f^{(n)}_{{\T^{\beta}}}(t)
&=&\mathcal{M}_{R}(\mathfrak{c}_{\alpha}+n+1)\frac{\E_x\left[{\textrm T}^{ \mathfrak{c}_{\alpha}}\right]}{\mathfrak{c}_{\alpha}\phi_{\alpha}^{+\prime}(-\mathfrak{c}_{\alpha}^+)} \sum_{k=0}^{\overline{{\rm{N}}}_{\Psi_\alpha}} {\overline{{\rm{N}}}_{\Psi_\alpha} \choose k}
\frac{\Gamma(\overline{{\rm{N}}}_{\Psi_\alpha}-n+1)}{\Gamma(k-n+1)} (-1)^{k}(1+\mathfrak{c}_{\alpha})_{k} \\
&=&\frac{\Gamma(\mathfrak{c}_{\alpha}+n+1-\overline{{\rm{N}}}_{\Psi_\alpha})}{\Gamma(\mathfrak{c}_{\alpha}+1)} \frac{\E_x\left[{\T}^{ \mathfrak{c}_{\alpha}}\right]}{\mathfrak{c}_{\alpha}\phi_{\alpha}^{+\prime}(-\mathfrak{c}_{\alpha}^+)} \sum_{k=0}^{\overline{{\rm{N}}}_{\Psi_\alpha}} {\overline{{\rm{N}}}_{\Psi_\alpha} \choose k}
\frac{\Gamma(\overline{{\rm{N}}}_{\Psi_\alpha}-n+1)}{\Gamma(k-n+1)} (-1)^{k}(1+\mathfrak{c}_{\alpha})_{k}
\end{eqnarray*}
which after rearranging the terms complete the proof.

\subsection{Proof of Proposition \ref{cor:sn}}\label{sec:proofp}
 Let $\phi \in \B$, denote by  $\varrho$ its associated subordinator and  write
 \[ \textrm{I}_{\phi}=\int_{0}^{\infty}e^{-\varrho_t}dt. \]
Then, for any $\Re(z)>0$, we have
\begin{equation*}
  \mathcal{M}_{\textrm{I}_{\phi}}(z)= \frac{\Gamma(z+1)}{W_{\phi}(z+1)}
\end{equation*}
see e.g.~\cite{PS_spectral}  and recalling that $\mathbb{F}_{\beta}$ stands for the Fr\'echet random variable of parameter $\beta>0$, we have, for any $\Re(z)<\beta$,
\begin{equation*}
  \mathcal{M}_{\mathbb{F}_{\beta}}(z)= \Gamma\left(-\frac{z}{\beta}+1\right).
\end{equation*}
Hence, by (shifted) Mellin transform identification, we deduce, from \eqref{idT_1}, that for any $\phi \in \B$ and $\beta>0$
\[ \mathbb{F}_{\beta}(\phi) \stackrel{(d)}{=}  \mathbb{F}_{\beta}\times \textrm{I}^{\frac{1}{\beta}}_{\phi}. \]
 Next, performing a change of variables yields that its Mellin transform takes the form
\begin{equation}\label{eq:mel}
\mathcal{M}_{\mathbb{F}_{\beta}(\phi)}(z-1)=
\frac{\Gamma(-\frac{z}{\beta}+1+\frac{1}{\beta})\Gamma(\frac{z}{\beta}+1-\frac{1}{\beta})}{W_{\phi}(\frac{z}{\beta}+1-\frac{1}{\beta})}.
\end{equation}
The proof of Theorem \ref{thm1}  combined with the analyticity of the gamma function to the right-half
plane entails that the mapping $z\mapsto \mathcal{M}_{\mathbb{F}_{\beta}(\phi)}(z-1)$ is analytical on the strip  $S_{\beta}=\{z\in \C; \: 1-\beta(\mathfrak{a}_\phi\mathbb{I}_{\{\phi(0)=0\}}+1) <\Re(z)<1+\beta\}$ and for any $\epsilon>0$ and $z=a+ib \in S_{\beta}$,
\begin{equation}\label{eq:mel1}
\left|\mathcal{M}_{\mathbb{F}_{\beta}(\phi)}(z-1)\right|\leq  e^{-|b|\frac{\Theta_\epsilon(\phi)}{\beta}}
\end{equation}
where  $\Theta_\epsilon(\phi) =\pi-\overline{\Theta}_{\phi}-\epsilon\geq \frac{\pi}{2}-\epsilon$ and recall that $
\overline{\Theta}_{\phi} = \limsup_{b\to \infty}
\frac{\int_0^{|b|}\arg \phi(1+iu)du}{|b|}$. Hence according to the theory of Mellin transforms, the law of $\mathbb{F}_{\beta}(\phi)$ is absolutely continuous with  a density $f_{\mathbb{F}_{\beta}(\phi)} \in \mathtt{C}^{\infty}_0\!\left(\mathbb{R}^+ \right)$ and which is analytical on the sector $S_{\phi}=\{z\in \C;\: |\arg(z)|<\pi-\overline{\Theta}_{\phi}\}$ and admits the Mellin Barnes representation
\begin{eqnarray*}\label{eq:MB}
  f_{\mathbb{F}_{\beta}(\phi)}(t)&=&\frac{1}{2\pi i} \int_{c-i\infty}^{c+i\infty} t^{-z}
\frac{\Gamma(-\frac{z}{\beta}+1+\frac{1}{\beta})\Gamma(\frac{z}{\beta}+1-\frac{1}{\beta})}{W_{\phi}(\frac{z}{\beta}+1-\frac{1}{\beta})} dz \\
&=& \frac{\beta}{2\pi i} \int_{\frac{c}{\beta}-i\infty}^{\frac{c}{\beta}+i\infty} t^{-\beta z}
\frac{\Gamma(-z+1+\frac{1}{\beta})\Gamma(z+1-\frac{1}{\beta})}{W_{\phi}(z+1-\frac{1}{\beta})} dz,
\end{eqnarray*}
which is absolutely integrable on $S_{\phi}$ for any $c \in S_{\beta}$. An application of Cauchy Theorem, see \cite[Proof of Lemma 8.16]{PS_spectral} for the details of similar arguments, gives that
\begin{eqnarray*}
  f_{\mathbb{F}_{\beta}(\phi)}(t)&=& \beta  t^{-\beta -1}\sum_{n=0}^{\infty} \frac{\Gamma(n+2)}{W_{\phi}(n+2)} \frac{(-t^{-\beta})^n}{n!}= \beta  t^{-\beta -1}\sum_{n=0}^{\infty} \frac{n+1}{\phi(n+1)}\frac{(-t^{-\beta})^n}{W_{\phi}(n+1)} \end{eqnarray*}
where, for the last identity, the recurrence relation $W_{\phi}(n+2)=\phi(n+1)W_{\phi}(n+1)$ and the one the gamma function is used in the last equality and to get that the series is  convergent for $|t|^{-\beta}<\phi(\infty)$.

\subsection{Proof of Theorem \ref{cor:sn}}
 First, recall that since   $\Psi \in \mathcal{N}_1$ then $\phi_{\alpha}^+\in \B_1$ and thus $\phi_{\beta}(u)=\phi(\beta u) =\mathcal{S}_{\phi_{\alpha}^+}(u)=\frac{u}{u+1}\phi_{\alpha}^+(u) \in \B$. Moreover, since \eqref{eq:ass} holds for $\phi_{\alpha}^+$, we have that either $d_{\alpha}^+>0 \textrm{ or }  \vartheta_{\alpha}^+(0,1)=\infty$. Thus,   as from e.g.~\cite[Proposition 4.1(3)]{PS_spectral}, with the obvious notation, \[d_{\beta}=\lim_{u\to \infty}\frac{\phi_{\beta}(u)}{u}=\lim_{u\to \infty}\frac{u}{u+1}\frac{\phi_{\alpha}^+(u)}{u}=d_{\alpha}^+\] and if $\vartheta_{\alpha}^+(0,1)=\infty$ with hence $d_{\beta}=d_{\alpha}^+=0$, then \[\vartheta_{\alpha}^+(0,1)=\lim_{u\to \infty}\phi_{\alpha}^+(u)=\lim_{u\to \infty}\frac{u}{u+1}\phi_{\alpha}^+(u)=\lim_{u\to \infty}\phi_{\beta}(u)=\vartheta_{\beta}(0,1).\]
  Consequently $\phi_{\beta}=\mathcal{S}_{\phi_{\alpha}^+}$ satisfies the condition  \eqref{eq:ass}.
 On the other hand we have \begin{equation}\label{eq:ppi}
 \phi'_{\beta}(0^+)=\beta  \phi'(0^+)=\phi_{\alpha}^+(0)\end{equation} which follows easily from the definition of $\phi_{\beta}$ and from \cite[Proposition 4.1(4.4)]{PS_spectral} that gives that for all $u\geq 0$, $0\leq u (\phi_{\alpha}^+)'(u)\leq \phi_{\alpha}^+(u)-\phi_{\alpha}^+(0)$ and hence $\lim_{u\to 0^+}u \:(\phi_{\alpha}^+)'(u)=0$. Putting pieces together we deduce that $\phi_{\beta} =\mathcal{S}_{\phi_{\alpha}^+} \in \B_{\varrho}$.
 Next, from the definition of $W_{\phi_{\alpha}^+}$  in \eqref{eq:functional-equation-for-W_phi}, as $\phi_{\alpha}^+\in \B_1\subset \B$, and writing $\overline{W}(u)=\frac{1}{u}W_{\phi_{\alpha}^+}(u), u>0$, we have that $\overline{W}(1)=1$ and
 \[ \overline{W}(u+1)= \frac{\phi^{+}_{\alpha}(u)}{u+1} W_{\phi_{\alpha}^+}(u)= \phi_{\beta}(u)\frac{W_{\phi_{\alpha}^+}(u)}{u}=\phi_{\beta}(u)\overline{W}(u).  \]
 Thus invoking the uniqueness argument given in \cite{Webster} yields that $ \overline{W}(u)=W_{\phi_{\beta}}(u)$, that is
 \[ W_{\phi_{\beta}}(u)=\frac{1}{u}W_{\phi_{\alpha}^+}(u).\]
  Hence, we deduce from \eqref{idT_1} and an application of the recurrence relation of the gamma function that, for any $-\beta<\Re(z)<0$,
  \begin{eqnarray}
   \nonumber 
    \mathcal{M}_{\T_{\Psi_{\alpha}}(\mathcal{S}_{\phi_{\alpha}^+})}(z) &=& x^{\frac{\alpha}{\beta}  z}\frac{ \phi_{\alpha}^+(0)}{\beta\phi'(0^+)} \frac{\Gamma(-\frac{z}{\beta})}{W_{\phi_{\beta}}(-\frac{z}{\beta})}\frac{\Gamma(\frac{z}{\beta}+1)W_{\phi_{\alpha}^+}(-\frac{z}{\beta})}{W_{\phi_{\alpha}^{-}}(\frac{z}{\beta}+1)}\\ &=&  x^{\frac{\alpha}{\beta}  z} \frac{\Gamma(1-\frac{z}{\beta}) \Gamma(\frac{z}{\beta}+1)}{W_{\phi_{\alpha}^{-}}(\frac{z}{\beta}+1)} \label{eq:idTF}
    \end{eqnarray}
where we used the identity  \eqref{eq:ppi}. Comparing this expression with \eqref{idT_2} yields, by uniqueness of the Mellin transform,  the first identity in law \eqref{eq:TF1}. Next, set $\beta=1$, from \cite[Theorem 2.4(1)]{PS_spectral}, we get that $\T_{\Psi_{\alpha}}(\mathcal{S}_{\phi_{\alpha}^+}) \stackrel{(d)}{=} x^{\alpha}\int_{0}^{\infty}e^{\alpha \overline{Y}_t} dt$ with $\overline{Y}=(\overline{Y}_t)_{t\geq0}$  a spectrally negative L\'evy process with characteristic exponent $\psi(z)=\frac{1}{\alpha}(z-\alpha)\phi^{-}(z)$ where we used the fact that $\phi_{\alpha}^{+}(0) W_{\phi_{\alpha}^{+}}(z)=\Gamma(1+z)$, that is $\phi_{\alpha}^{+}(u)=u+1$. The second claim follows since, from \eqref{eq:id_T_exp}, we also have ${\rm{T}}_\psi\stackrel{(d)}{=} x^{\alpha}\int_{0}^{\infty}e^{\alpha \overline{Y}_t} dt$.
Finally if  now $\Psi \in \mathcal{N}^-_1=\{\Psi \in \mathcal{N}_1 \textrm{ with } \phi_{\alpha}^{-}(u)=\alpha u\}$, we have $W_{\phi_{\alpha}^{-}}(\frac{z}{\beta}+1)= \alpha^\frac{z}{\beta}\Gamma(\frac{z}{\beta}+1)$ and thus  \eqref{eq:idTF} entails that in this case
\[\mathcal{M}_{\T_{\Psi_{\alpha}}(\mathcal{S}_{\phi_{\alpha}^+})}(z)=  (\alpha x^{\alpha})^{\frac{z}{\beta}} \Gamma\left(1-\frac{z}{\beta}\right)\]
which completes the proof after mentioning that the self-decomposability  property of the Fr\'echet distribution is found, for $0<\beta\leq 1$,  in \cite{Bondesson}  and  for $\beta>1$, in \cite[Lemma 1]{KP}.

\bibliographystyle{plain}

 \end{document}